\documentclass[12pt,twoside,english]{article}
\usepackage[T1]{fontenc}
\usepackage[latin9]{inputenc}
\usepackage{amsthm}
\usepackage{amsmath}
\usepackage{amssymb}

\makeatletter
\numberwithin{equation}{section}
\numberwithin{figure}{section}
\usepackage{enumitem}		% customizable list environments
\theoremstyle{plain}
\newtheorem{thm}{\protect\theoremname}
  \theoremstyle{definition}
  \newtheorem{defn}[thm]{\protect\definitionname}
  \theoremstyle{plain}
  \newtheorem{prop}[thm]{\protect\propositionname}
  \theoremstyle{plain}
  \newtheorem{lem}[thm]{\protect\lemmaname}
  \theoremstyle{plain}
  \newtheorem{cor}[thm]{\protect\corollaryname}

\@ifundefined{date}{}{\date{}}
\makeatother

\usepackage{babel}
  \providecommand{\corollaryname}{Corollary}
  \providecommand{\definitionname}{Definition}
  \providecommand{\lemmaname}{Lemma}
  \providecommand{\propositionname}{Proposition}
\providecommand{\theoremname}{Theorem}

\begin{document}

\title{\textbf{New Douglas-Rachford Algorithmic Structures and Their Convergence
Analyses}}

\author{Yair Censor and Rafiq Mansour}

\maketitle
\begin{center}
Department of Mathematics 
\par\end{center}

\begin{center}
University of Haifa 
\par\end{center}

\begin{center}
Mt. Carmel, Haifa 3498838, Israel\medskip{}

\par\end{center}

\begin{center}
December 24, 2014. Revised: June 23, 2015, September 2, 2015, October
20, 2015 and November 25, 2015.
\par\end{center}
\begin{abstract}
In this paper we study new algorithmic structures with Douglas-Rachford
(DR) operators to solve convex feasibility problems. We propose to
embed the basic two-set-DR algorithmic operator into the String-Averaging
Projections (SAP) and into the Block-Iterative Projection (BIP) algorithmic
structures, thereby creating new DR algorithmic schemes that include
the recently proposed cyclic Douglas-Rachford algorithm and the averaged
DR algorithm as special cases. We further propose and investigate
a new multiple-set-DR algorithmic operator. Convergence of all these
algorithmic schemes is studied by using properties of strongly quasi-nonexpansive
operators and firmly nonexpansive operators.
\end{abstract}

\section{Introduction\label{sect:introduction}}

\textbf{Contribution of this paper}. We study new algorithmic structures
for the Douglas-Rachford (DR) algorithm%
\footnote{We use the terms ``algorithm'' and ``algorithmic structures''
for the iterative processes studied here although no termination criteria
are present and only the asymptotic behavior of these processes is
studied.%
}. Our starting points for the developments presented here are the
two-set-DR original algorithm and the recent cyclic-DR algorithm of
\cite{btam14}, designed to solve convex feasibility problems. They
use the same basic algorithmic operator which reflects the current
iterate consecutively into two sets and takes the midpoint between
the current iterate and the end-point of the two consecutive reflections
as the next iterate.

The convex feasibility problem (CFP) is to find an element $x\in C$
where $C_{i}$, $i\in I:=\left\{ 1,2,\ldots,m\right\} $, form a finite
family of closed convex sets in a Hilbert space $\mathcal{H},$ $C:=\bigcap_{i\in I}C_{i}$
and $C\neq\textrm{Ø}.$ There are many algorithms in the literature
for solving CFPs, see, e.g., \cite{bb96}. In particular, two algorithmic
structures that encompass many specific feasibility-seeking algorithms
are the String-Averaging Projections (SAP) method \cite{ceh01} and
the Block-Iterative Projection (BIP) method \cite{ac89}. 

In this paper we do two things: (i) create new algorithmic structures
with the 2-set-DR algorithmic operator, and (ii) define and study
an ``$m$-set-DR operator''. 

First we employ the two-set-DR algorithmic operator and embed it in
the SAP and BIP algorithmic structures. In doing so one obtains two
new families of DR algorithms, of which the two-set-DR original algorithm
and the recent cyclic-DR algorithm and the averaged DR algorithm are
special cases. Convergence analyses of these new DR algorithms are
provided.

In our String-Averaging Douglas-Rachford (SA-DR) scheme, we separate
the index set $I$ of the CFP into subsets, called ``strings'',
and proceed along each string by applying the basic two-set-DR operator
(this can be done in parallel for all strings) sequentially along
the string. Then a convex combination of the strings' end-points is
taken as the next iterate.

In our Block-Iterative Douglas-Rachford (BI-DR) scheme, the index
set $I$ of the CFP is again separated into subsets, now called ``blocks'',
dividing the family of convex closed sets into blocks of sets. The
basic two-set-DR operator is applied in a specific way to each block
and the algorithmic scheme proceeds sequentially over the blocks.

Finally, we propose and investigate a generalization of the 2-set-DR
algorithmic operator itself. Instead of reflecting the current iterate
consecutively into two sets and taking the midpoint between the current
iterate and the end-point of the two consecutive reflections as the
next iterate, we propose to allow the algorithmic operator to perform
a finite number, say $r$ (greater or equal 2), of consecutive reflections
into $r$ sets and only then take the midpoint between the current
iterate and the end-point of the $r$ consecutive reflections as the
next iterate. We show how this ``$m$-set-DR operator'' works algorithmically.

We study the convergence of all algorithmic schemes, under the assumption
that $C$ or its interior are nonempty by using properties of strongly
quasi-nonexpansive operators and firmly nonexpansive operators. In
particular, a corner stone of our results is the recognition that
the 2-set-DR operator is not only firmly nonexpansive, thus nonexpansive,
as stated in \cite[Fact 2.2]{btam14}, but also strongly quasi-nonexpansive.

\textbf{The framework. }Since a reflection is a nonexpansive operator
and the class of nonexpansive operators is closed under composition,
the 2-set-DR operator is an averaged operator. Averaged operators
form a very nice class since they are closed under compositions and
convex combinations. Therefore, since all operators discussed in the
paper are averaged one can get alternative proofs of our results from
{[}26{]} or {[}7, Sections 5.2 and 5.3{]}. We chose, however, to work
within the framework of various ``quasi'' operators recognizing
the generality of this framework.

\textbf{Current literature}. Current literature witnesses a strong
interest in DR algorithms in the framework of splitting methods for
optimization, see, e.g., \cite{combettes2011proximal}. We are interested
in the DR algorithm for the feasibility problem and in this direction
there are several relevant publications that include also applications
in various fields. An analysis of the behavior of the cyclic Douglas\textendash Rachford
algorithm when the target intersection set is possibly empty was undertaken
in \cite{borwein2013cyclic}, consult this paper also for many additional
relevant references. The work in \cite{borwein2014reflection} applies
the DR algorithm to the problem of protein conformation. Recent positive
experiences applying convex feasibility algorithms of Douglas-Rachford
type to highly combinatorial and far from convex problems appear in
\cite{artacho2013recent}. Systematic investigation of the asymptotic
behavior of averaged alternating reflections (AAR) which are also
known as the 2-set-DR operators, in the general case when the sets
do not necessarily intersect can be found in \cite{bauschke2004finding},
and \cite{bauschke2006strongly} presents a new averaged alternating
reflections method which produces a strongly convergent sequence.
General recommendations for successful application of the DR algorithm
to convex and nonconvex real matrix-completion problems are presented
in \cite{artacho2013douglas}.

For the convergence of the DR algorithm under various constraints,
see, e.g., \cite{svaiter2011weak} which proves that any sequence
generate by the DR algorithm converges weakly to a solution of an
inclusion problem. In \cite{bauschke2014local} the authors prove
the two-set-DR algorithm's local convergence to a fixed point when
the two sets are finite unions of convex sets, while \cite{borwein2011douglas}
provides convergence results for a prototypical nonconvex two-sets
scenario in which one of the sets is a Euclidean sphere. The results
in \cite{artacho2013global} establish a region of convergence for
the prototypical non-convex Douglas-Rachford iteration which finds
a point in the intersection of a line and a circle. The work in \cite{bauschke2015linear}
introduces regularity notions for averaged nonexpansive operators,
and obtains linear and strong convergence results for quasi-cyclic,
cyclic, and random iterations. New convergence results on the Borwein\textendash Tam
method (BTM) which is also known as the cyclic DR algorithm, and on
the cyclically anchored Douglas\textendash Rachford algorithm (CADRA)
are also presented. 

Relevant to our analyses are properties of operators under compositions
and convex combinations. In, e.g., \cite{bauschke2003composition}
one learns that the composition of projections onto closed convex
sets in Hilbert space is asymptotically regular, and \cite{bauschke2012compositions}
proves that compositions and convex combinations of asymptotically
regular firmly nonexpansive mappings are also asymptotically regular.
In \cite{combettes2004solving} a unified fixed point theoretic framework
is proposed to investigate the asymptotic behavior of algorithms for
finding solutions to monotone inclusion problems. The basic iterative
scheme under consideration involves non-stationary compositions of
perturbed averaged nonexpansive operators.

For details on BIP, SAP and other projection methods, see, e.g., \cite{butnariu1990behavior}
which studies the behavior of a class of BIP algorithms for solving
convex feasibility problems. In \cite{byrne1995block} the simultaneous
MART algorithm (SMART) and the expectation maximization method for
likelihood maximization (EMML) are extended to block-iterative versions,
called BI-SMART and BI-EMML, respectively, that converge to a solution
in the feasible case. The work in \cite{censor2002block} formulates
a block-iterative algorithmic scheme for the solution of systems of
linear inequalities and/or equations and analyze its convergence.
The excellent review \cite{bb96} discusses projection algorithms
for solving convex feasibility problems. More recently, \cite{censor2003convergence}
discusses the convergence of string-averaging projection schemes for
inconsistent convex feasibility problems, \cite{censor2009string}
proposes a definition of sparseness of a family of operators and investigates
a string-averaging algorithmic scheme that favorably handles the common
fixed points problem when the family of operators is sparse.

\textbf{Potential computational advantages.} We have no computational
experience with the new DR algorithms proposed and studied here. Comparative
computational performance can really be made only with exhaustive
testing of the many possible specific variants of the new DR algorithms
permitted by the general schemes and their various user-chosen parameters.
The computational advantages of string-averaging and block-iterative
algorithmic structures have been shown in the past for algorithms
that use orthogonal projections%
\footnote{The term ``orthogonal projection'' is used here not only for the
case when the sets are subspaces but in general. It means here ``nearest
point projection'' operator or ``metric projection'' operator.%
} rather than DR operators. For example, the work on proton computed
tomography (pCT) in \cite{karonis2013distributed} employs very efficiently
a parallel code that uses a version of the string-averaging algorithm
called component-averaged row projections (CARP), see \cite{gordon2005component},
and a version of a block-iterative algorithm called diagonally-relaxed
orthogonal projections (DROP), see \cite{censor2008diagonally}.

It is plausible to hypothesize that since the string-averaging algorithmic
structure and the cyclic DR algorithm of \cite{btam14} have been
demonstrated to be computationally useful separately then so might
very well be their combination in the String-Averaging Douglas-Rachford
(SA-DR). Admittedly, these practical questions should be resolved
in future work, preferably within the context of a significant real-world
application.

\textbf{Structure of the paper}. The paper is organized as follows:
In Section \ref{sec:preliminaries}, we give definitions and preliminaries.
In Section \ref{sect:sa-and-bi}, we introduce the new String-Averaging
Douglas-Rachford (SA-DR) scheme and the new Block-Iterative Douglas-Rachford
(BI-DR) scheme. For both the SA-DR and the BI-DR algorithms we prove
strong convergence to a point in the intersection of the sets under
the assumption that $\textup{int}\:\bigcap_{i\in I}C_{i}\neq\textrm{Ø}$.
In Section \ref{sect:third-generalization} we study our new ``$m$-set-DR
operator'' and show how it works.

\section{Preliminaries\label{sec:preliminaries}}

For the reader's convenience we include in this section some properties
of operators in Hilbert space that will be used to prove our results.
We use the recent excellent book of Cegielski \cite{Ceg-book} as
our desk-copy in which all the results of this section can be found
\cite[Chapter 2]{Ceg-book}. Let $\mathcal{H}$ be a real Hilbert
space with inner product $\left\langle \cdot,\cdot\right\rangle $
and induced norm $\parallel\cdot\parallel$, and let $X\subseteq\mathcal{H}$
be a closed convex subset.\renewcommand{\labelenumi}{\roman{enumi}.}
\begin{defn}
\label{def:fejer}An operator $T:X\rightarrow\mathcal{H}$ is:
\begin{enumerate}
\item \noindent \textit{Fejér monotone }(FM) with respect to a nonempty
subset $C\subseteq X$, if $\Vert T(x)-z\Vert\leq\Vert x-z\Vert$
for all $x\in X$ and $z\in C$.
\item \emph{Strictly Fej}\textit{é}\emph{r monotone} (sFM) with respect
to a nonempty subset $C\subseteq X$, if $\Vert T(x)-z\Vert<\Vert x-z\Vert$
for all $x\notin C$ and $z\in C$.
\item \emph{Nonexpansive }(NE)\emph{,} if $\Vert T(x)-T(y)\Vert\leq\Vert x-y\Vert$
for all $x,y\in X$.
\item \emph{Firmly nonexpansive} (FNE)\emph{,} if $\left\langle T(x)-T(y),x-y\right\rangle \geq\Vert T(x)-T(y)\Vert^{2}$
for all $x,y\in X.$
\item \textit{Strongly nonexpansive}\textsl{ }(SNE), if $T$ is nonexpansive
and for all sequences $\left\{ x^{k}\right\} _{k=0}^{\infty}$, $\left\{ y^{k}\right\} _{k=0}^{\infty}\subseteq X$
such that $\left(x^{k}-y^{k}\right)$ is bounded and $\Vert x^{k}-y^{k}\parallel-\parallel T(x^{k})-T(y^{k})\parallel\rightarrow0$
it follows that $\left(x^{k}-y^{k}\right)-\left(T(x^{k})-T(y^{k})\right)\rightarrow0$. 
\end{enumerate}
\end{defn}
\noindent The following proposition is well-known, see, e.g., \cite[Theorem 2.2.4]{Ceg-book}
and \cite[Corollary 2.2.20]{Ceg-book}.
\begin{prop}
\label{prop:fne-is-ne}Every firmly nonexpansive operator is nonexpansive,
and any convex combination of firmly nonexpansive operators is firmly
nonexpansive.\end{prop}
\begin{defn}
\label{def:QNE}An operator $T:X\rightarrow\mathcal{H}$ having a
fixed point is:
\begin{enumerate}
\item \noindent \emph{Quasi-nonexpansive} (QNE), if $T$ is Fejér monotone
with respect to the fixed points set Fix$T$, i.e., if $\Vert T(x)-z\Vert\leq\Vert x-z\Vert$
for all $x\in X$ and $z\in$Fix$T$.
\item \emph{Strictly quasi-nonexpansive} (sQNE), if $T$ is strictly Fejér
monotone with respect to Fix$T$, i.e., if $\Vert T(x)-z\Vert<\Vert x-z\Vert$
for all $x\notin$Fix$T$ and $z\in$Fix$T$.
\item \emph{C-strictly quasi-nonexpansive} (C-sQNE), where $C\neq\textrm{Ø}$
and $C\subseteq$Fix$T$, if $T$ is quasi-nonexpansive and $\Vert T(x)-z\Vert<\Vert x-z\Vert$
for all $x\notin$Fix$T$ and $z\in C$.
\item \emph{$\alpha$-strongly quasi-nonexpansive} ($\alpha$-SQNE), if
$\Vert T(x)-z\Vert^{2}\leq\Vert x-z\Vert^{2}-\alpha\Vert T(x)-x\Vert^{2}$
for all $x\in X$ and $z\in$Fix$T$, where $\alpha\geq0$. If $\alpha>0$
then $T$ is called \emph{strongly quasi-nonexpansive} (SQNE).
\end{enumerate}
\end{defn}
The following implications follow directly from the definitions, see
\cite[page 47]{Ceg-book} and \cite[Remark 2.1.44(iii)]{Ceg-book}.
\begin{prop}
\label{prop:sQNE}For an operator $T:X\rightarrow\mathcal{H}$ having
a fixed point, the following statements hold: 
\begin{enumerate}
\item If $T$ is sQNE then $T$ is $C$-sQNE, where $C\subseteq$\textup{Fix}$T$. 
\item If $T$ is \textup{Fix}$T$-sQNE then $T$ is sQNE.
\item If $T$ is strongly quasi-nonexpansive then it is strictly quasi-nonexpansive.
\end{enumerate}
\end{prop}
From \cite[Theorem 2.3.5]{Ceg-book}, see also \cite[Fig. 2.14]{Ceg-book},
any composition of SNE operators is SNE, and any convex combination
of SNE operators is SNE. 
\begin{thm}
\label{thm:1}\cite[Corollary 2.1.42]{Ceg-book} Let $U_{i}:X\rightarrow\mathcal{H}$,
$i\in I,$ be quasi-nonexpansive with $C:=\bigcap_{i\in I}\textup{Fix}U_{i}\neq\textrm{Ø}$
and let: $U:=U_{m}U_{m-1}\cdots U_{1}$. If $\textup{int}\, C\neq\textrm{Ø}$,
then $\textup{Fix}U=\bigcap_{i\in I}\textup{Fix}U_{i}$ and $U$ is
\textup{int}C-strictly quasi-nonexpansive. 
\end{thm}
Denoting by $\vartriangle_{m}:=\{u\in R^{m}\mid u\geq0,$ $\sum_{i=1}^{m}u_{i}=1\}$
the standard simplex, and by \textit{$ri\vartriangle_{m}$} its relative
interior, a function \textit{$w:X\rightarrow\vartriangle_{m},$ with
$w(x)=(w_{1}(x),w_{2}(x),\ldots,w_{m}(x))$ is called a weight function}.
According to \cite[Definition 2.1.25 and text on Page 50]{Ceg-book}
a weight function $w:X\rightarrow ri\vartriangle_{m}$ is \emph{appropriate}
with respect to any family of operators $\left\{ U_{i}\right\} _{i\in I}$
if: (i) $w\in ri\vartriangle_{m}$ is a vector of constant weights,
or if (ii) $w_{i}(x)>0$ for all $x\notin\textup{Fix}U_{i}$ and all
$i\in I$. Throughout the paper, we use only the first option of constant
weights.

We have also the following.
\begin{thm}
\label{thm:2}\cite[Theorem 2.1.26]{Ceg-book} Let the operators $U_{i}:X\rightarrow X$
, $i\in I$, with $\bigcap_{i\in I}$\textup{Fix}$U_{i}\neq\textrm{Ø}$,
be $C$-strictly quasi-nonexpansive, where $C\subseteq\bigcap_{i\in I}$\textup{Fix}$U_{i}$,
$C\neq\textrm{Ø}$. If $U$ has one of the following forms:
\begin{enumerate}
\item $U:=\sum_{i\in I}w_{i}U_{i}$ and the weight function $w:X\rightarrow\vartriangle_{m}$
is appropriate,
\item $U:=U_{m}U_{m-1}\ldots U_{1},$
\end{enumerate}
\end{thm}
\begin{flushleft}
\textit{then }
\par\end{flushleft}

\textit{
\begin{equation}
\textnormal{Fix}U=\bigcap_{i\in I}\textnormal{Fix}U_{i},
\end{equation}
}

\noindent \textit{and $U$ is $C$-strictly quasi-nonexpansive.}
\begin{thm}
\label{thm:3}\cite[Theorem 2.1.14]{Ceg-book} Let $U_{i}:X\rightarrow\mathcal{H},$
$i\in I:=\left\{ 1,2,\ldots,m\right\} ,$ be nonexpansive operators
with a common fixed point and let $U:=\sum_{i\in I}w_{i}U_{i}$ with
the weight function $w(x)=(w_{1}(x),w_{2}(x),\ldots,w_{m}(x))\in ri\vartriangle_{m}$.
Then
\end{thm}
\noindent \textit{
\begin{equation}
\textnormal{Fix}U=\bigcap_{i\in I}\textnormal{Fix}U_{i}.
\end{equation}
}
\begin{defn}
An operator $U:X\rightarrow X$ is \textit{asymptotically regular}
if for all $x\in X$,
\end{defn}
\noindent 
\begin{equation}
\underset{k\rightarrow\infty}{\textnormal{lim}}\Vert U^{k+1}(x)-U^{k}(x)\Vert=0.
\end{equation}

By combining the results of \cite[Theorem 3.4.3]{Ceg-book}, \cite[Corollary 3.4.6]{Ceg-book}
for $\lambda=1$, and \cite[Theorem 3.4.9]{Ceg-book}, see also \cite[Fig. 3.2]{Ceg-book},
we can state the following.
\begin{thm}
\label{thm:4} Let $U:X\rightarrow X$ be an operator with a fixed
point. If $U$ is strongly quasi-nonexpansive or strongly nonexpansive,
then $U$ is asymptotically regular.\end{thm}
\begin{defn}
\label{def:demi-closed}An operator $T:X\rightarrow\mathcal{H}$ is
\textit{demi-closed }at 0 if for any sequence $x^{k}\rightharpoonup y\in X$
with $T(x^{k})\rightarrow0$ we have $T(y)=0$.
\end{defn}
If we replace the weak convergence $x^{k}\rightharpoonup y$ by the
strong one in Definition \ref{def:demi-closed}, then we obtain the
definition of the closedness of $T$ at 0. If $\mathcal{H}$ is finite-dimensional,
then the notions of a demi-closed operator and a closed operator coincide,
see, e.g., \cite[Page 107]{Ceg-book}. 

Denoting by Id the identity operator we have the following theorem.
\begin{thm}
\label{thm:5}\cite[Theorem 3.5.2]{Ceg-book} Let $X\subseteq\mathcal{H}$
be nonempty closed convex subset of a finite-dimensional Hilbert space
$\mathcal{H}$ and let $U:X\rightarrow X$ be an operator with a fixed
point and such that $U-\textup{Id}$ is closed at 0. If $U$ is quasi-nonexpansive
and asymptotically regular, then, for arbitrary $x\in X$, the sequence
$\left\{ U^{k}(x)\right\} _{k=0}^{\infty}$ converges to a point $z\in\textup{Fix}U$.
\end{thm}
Finally, here is the well-known theorem due to Opial.
\begin{thm}
\label{thm:6-opial}\cite[Theorem 3.5.1]{Ceg-book} Let $X\subseteq\mathcal{H}$
be a nonempty closed convex subset of a Hilbert space $\mathcal{H}$
and let $U:X\rightarrow X$ be a nonexpansive and asymptotically regular
operator with a fixed point. Then, for any $x\in X$, the sequence
$\left\{ U^{k}(x)\right\} _{k=0}^{\infty}$ converges weakly to a
point $z\in\textup{Fix}U$.
\end{thm}

\section{String-Averaging and Block-Iterative Douglas-Rachford \label{sect:sa-and-bi}}

In this section we describe our new String-Averaging Douglas-Rachford
(SA-DR) algorithmic scheme and our new Block-Iterative Douglas-Rachford
(BI-DR) algorithmic scheme and prove their convergence. Let $C_{1},C_{2},\ldots,C_{m}$,
be nonempty closed convex subsets of $\mathcal{H}$, defining a convex
feasibility problem (CFP) of finding an element in their intersection.
We call $I:=\{1,2,\ldots,m\}$ the index set of the convex feasibility
problem.

\subsection{The algorithms}
\begin{defn}
\label{def:relax-and-reflect-1}Let $T:X\rightarrow\mathcal{H}$ and
$\lambda\in[0,2]$. The operator $T_{\lambda}:X\rightarrow\mathcal{H}$
defined by $T_{\lambda}:=(1-\lambda)\textup{Id}+\lambda T$ is called
a \emph{$\lambda$-relaxation} or, in short,\emph{ relaxation} of
the operator $T$. If $\lambda=2,$ then $T_{\lambda}$ is called
the \textit{reflection} of $T$.
\end{defn}
\noindent Note that if $\lambda=1$ then $T_{\lambda}=T_{1}=T$. The\emph{
}2-set-Douglas-Rachford operator is defined as follows, see, e.g.,
\cite[Equation (2)]{btam14}.
\begin{defn}
\label{def:2-DR-1}Let $A,B\subseteq\mathcal{H}$ be closed convex
subsets and let $P_{A}$ and $P_{B}$ be the orthogonal projections
onto $A$ and $B$, respectively. The operators $R_{A}:=2P_{A}-\textup{Id}$
and $R_{B}:=2P_{B}-\textup{Id}$ are the reflection operators into
$A$ and $B$, respectively. The operator $T_{B,A}:\mathcal{H}\rightarrow\mathcal{H}$,
defined by,
\end{defn}
\noindent 
\begin{equation}
T_{B,A}:=\frac{1}{2}(\textup{Id}+R_{A}R_{B})
\end{equation}

\noindent is the \emph{``2-set-Douglas-Rachford''} (2-set-DR) operator
.

This operator was termed earlier ``averaged alternating reflection''
(AAR), see \cite[Subsection 4.3.5]{Ceg-book} for details and references.
For $t=1,2,\ldots,M$ , let the ``string'' $I_{t}$ be an ordered
finite nonempty subset of $I$ of the CFP, of the form

\noindent 
\begin{equation}
I_{t}=(i_{1}^{t},i_{2}^{t},\ldots,i_{\gamma(t)}^{t}),
\end{equation}

\noindent where the ``length'' of the string $I_{t}$, denoted by
$\gamma(t)$, is the number of elements in $I_{t}$. Denoting $T_{C_{i},C_{j}}$
by $T_{i,j}$, the String-Averaging Douglas-Rachford-(SA-DR) algorithmic
scheme is as follows.\medskip{}

\noindent \textbf{Algorithm 1. The String-Averaging Douglas-Rachford
Scheme.}

\noindent \textbf{Initialization}: $x^{0}\in\mathcal{H}$ is arbitrary.

\noindent \textbf{Iterative Step:} Given the current iterate $x^{k}$,

\noindent (i) Calculate, for all $t=1,2,\ldots,M$,

\begin{equation}
T_{t}(x^{k}):=T_{i_{\gamma(t)}^{t},i_{1}^{t}}T_{i_{\gamma(t)-1}^{t},i_{\gamma(t)}^{t}}\ldots T_{i_{2}^{t},i_{3}^{t}}T_{i_{1}^{t},i_{2}^{t}}(x^{k}),\label{eq:string}
\end{equation}

\noindent (ii) Calculate the convex combination of the strings' end-points
by

\noindent 
\begin{equation}
x^{k+1}=\sum_{t=1}^{M}w_{t}T_{t}(x^{k}),
\end{equation}

\noindent with $w_{t}>0,$ for all $t=1,2,\ldots,M$, and $\sum_{t=1}^{M}w_{t}=1$.
\medskip{}

\noindent Note that the work on the strings in part (i) of the Iterative
Step in the SA-DR scheme of Algorithm 1 can be performed in parallel
on all strings. 

\noindent In order to present the Block-Iterative Douglas-Rachford
scheme, we again look at nonempty subsets 
\begin{equation}
I_{t}=(i_{1}^{t},i_{2}^{t},\ldots,i_{\gamma(t)}^{t}),
\end{equation}
of the index set $I$ of the CFP, which are now called ``blocks''.

\noindent \medskip{}

\noindent \textbf{Algorithm 2. The Block-Iterative Douglas-Rachford
Scheme.}

\noindent \textbf{Initialization:} $x^{0}\in\mathcal{H}$ is arbitrary.

\noindent \textbf{Iterative Step:} Given the current iterate $x^{k}$,
pick a block index $t=t(k)$ according to a cyclic rule $t(k)=k\textup{mod}M+1$. 

\noindent (i) Calculate intermediate points obtained by applying the
2-set-Douglas-Rachford operator to pairs of sets in the $t$-th block
as follows. For all $\ell=1,2,\ldots,\gamma(t)-1$ define

\begin{equation}
z_{\ell}=T_{i_{\ell}^{t},i_{\ell+1}^{t}}(x^{k}),\label{eq:z-el-1}
\end{equation}

\noindent and for $\ell=\gamma(t)$ let

\begin{equation}
z_{\gamma(t)}=T_{i_{\gamma(t)}^{t},i_{1}^{t}}(x^{k}).\label{eq:z-el-2}
\end{equation}

\noindent (ii) Calculate the convex combination of the intermediate
points

\noindent 
\begin{equation}
x^{k+1}=\sum_{\ell=1}^{\gamma(t)}w_{\ell}^{t}z_{\ell},
\end{equation}

\noindent with $w_{\ell}^{t}>0,$ for all $\ell=1,2,\ldots,\gamma(t)$,
and $\sum_{\ell=1}^{\gamma(t)}w_{\ell}^{t}=1$.

\noindent \medskip{}

Note that the work on the pairs of sets in a block in part (i) of
the Iterative Step in the BI-DR scheme of Algorithm 2 can be performed
in parallel on all pairs.

\subsection{Convergence proofs}

The following lemma characterizes the fixed points of the 2-set Douglas-Rachford
operator.
\begin{lem}
\label{lem:fixed-points-2-RD}\cite[Corollary 3.9]{bauschke2004finding}
Let $A,B\subseteq\mathcal{H}$ be closed and convex with nonempty
intersection. Then
\end{lem}
\begin{equation}
P_{A}\textup{Fix}T_{A,B}=A\cap B.\label{eq:3.9}
\end{equation}

If the right-hand side of (\ref{eq:3.9}) has nonempty interior additional
information is available from \cite[Corollary 4.3.17(ii)]{Ceg-book}.
\begin{lem}
\label{lem:16}\cite[Corollary 4.3.17(ii)]{Ceg-book} Let $A,B\subseteq\mathcal{H}$
be closed and convex sets and let $T_{A,B}$ be their 2-set-DR operator.
If $\textup{int}\:(A\cap B)\neq\textrm{Ø}$ then $A\cap B=\textup{Fix}T_{A,B}$.
\end{lem}
A corner stone of our subsequent results is the recognition that the
2-set-DR operator is not only firmly nonexpansive, as stated in \cite[Fact 2.2]{btam14},
but also strongly quasi-nonexpansive.
\begin{lem}
\label{lem:2}Every 2-set-DR operator $T_{A,B}$ is:
\begin{enumerate}
\item Strongly quasi-nonexpansive (SQNE),
\item Strongly nonexpansive (SNE).
\end{enumerate}
\end{lem}
\begin{proof}
(i) Follows from \cite[Corollary 2.2.9]{Ceg-book}.

(ii) By \cite[Theorem 2.3.4]{Ceg-book} with $\lambda=1$ and $T=T_{1}$
we obtain that \textit{$T$ }is SNE.
\end{proof}
Our convergence theorems for the above two algorithmic schemes can
now be stated and proved.
\begin{thm}
\textbf{\label{thm:7}(SA-DR)}. Let $C_{1},C_{2},\ldots,C_{m}\subseteq\mathcal{H}$
be closed and convex sets that define a convex feasibility problem
(CFP). If 
\begin{equation}
\textup{int}\:\bigcap_{i\in I}C_{i}\neq\textrm{Ø}\label{eq:int}
\end{equation}
then for any $x^{0}\in\mathcal{H}$, any sequence $\left\{ x^{k}\right\} _{k=0}^{\infty}$,
generated by the SA-DR Algorithm 1, with strings such that $I=I_{1}\cup I_{2}\cup\ldots\cup I_{M}$,
converges strongly to a point $x^{*}\in\bigcap_{i\in I}C_{i}$.\end{thm}
\begin{proof}
From Lemma \ref{lem:2} we have that the 2-set-DR operators $T_{i_{\ell}^{t},i_{\ell+1}^{t}}$for
all $\ell=1,2,\ldots,\gamma(t)-1$ and $T_{i_{\gamma(t)}^{t},i_{1}^{t}}$
are strongly quasi-nonexpansive. Denoting the intersection of all
fixed points sets of all 2-set-DR operators within the $t$-th string
by 
\begin{equation}
\Gamma_{t}:=\left(\bigcap_{\ell=1}^{\gamma(t)-1}\textup{Fix}T_{i_{\ell}^{t},i_{\ell+1}^{t}}\right)\bigcap\textup{Fix}T_{i_{\gamma(t)}^{t},i_{1}^{t}},\label{eq:intersect-T}
\end{equation}
we first show that $\textup{int}\:\Gamma_{t}$ is nonempty. Each point
in the nonempty intersection $\bigcap_{i\in I}C_{i}$ is a fixed point
for any 2-set-DR operator with respect to any pair of sets from the
family of sets in the CFP, meaning that $\bigcap_{i\in I}C_{i}\neq\textrm{Ø}$
is included in the fixed points set of any of the operators that appear
in the right-hand side of (\ref{eq:intersect-T}). So we have 
\begin{equation}
\Gamma_{t}\supseteq\bigcap_{i=1}^{\gamma(t)}C_{i}\neq\textrm{Ø},
\end{equation}
and, by (\ref{eq:int}), 

\begin{equation}
\textup{int}\:\Gamma_{t}\neq\textrm{Ø}.
\end{equation}

Identifying the 2-set-DR operators in a string with the individual
operators in Theorem \ref{thm:1}, and since any SQNE operator is
also QNE, the conditions of Theorem \ref{thm:1} are met and we conclude
that any string operator $T_{t}$ in (\ref{eq:string}) is int$\Gamma_{t}$-strictly
quasi-nonexpansive, and 
\begin{equation}
\textup{Fix}T_{t}=\Gamma_{t},
\end{equation}
for all $t=1,2,\ldots,M$. Looking at the intersection of the fixed
point sets of all string operators 
\begin{equation}
\Gamma:=\bigcap_{t=1}^{M}\textup{Fix}T_{t},
\end{equation}
we conclude again, re-applying the previous considerations, that 
\begin{equation}
\Gamma\supseteq\bigcap_{i\in I}C_{i}\neq\textrm{Ø},
\end{equation}
and that for any subset $\Theta$ of $\textup{int}\:\Gamma$, the
operators $T_{t}$ are $\Theta$-strictly quasi-nonexpansive, since
they are int$\Gamma_{t}$-strictly quasi-nonexpansive. So, by Theorem
\ref{thm:2}(i), 
\begin{equation}
Fix\left(\sum_{t=1}^{M}w_{t}T_{t}\right)=\Gamma,\label{eq:3.18}
\end{equation}
and $\sum_{t=1}^{M}w_{t}T_{t}$ is $\textup{int}\:\Gamma$-strictly
quasi-nonexpansive, which in turn implies that it is a quasi-nonexpansive
operator. By Lemma \ref{lem:2}(ii) the 2-set-DR operators $T_{i_{\ell}^{t},i_{\ell+1}^{t}}$for
all $\ell=1,2,\ldots,\gamma(t)-1$ and $T_{i_{\gamma(t)}^{t},i_{1}^{t}}$
are strongly nonexpansive, therefore, their composition and the convex
combination of their compositions $\sum_{t=1}^{M}w_{t}T_{t}$ are
strongly nonexpansive, thus, nonexpansive, see Definition \ref{def:fejer}(v).
Using the demi-closedness principle embodied in \cite[Lemma 3.2.5]{Ceg-book}
and \cite[Definition 3.2.6]{Ceg-book} for the operator $\sum_{t=1}^{M}w_{t}T_{t}$,
the operator $(\sum_{t=1}^{M}w_{t}T_{t}-\textup{Id})$ is demi-closed
at 0. By Theorem \ref{thm:4}, $\sum_{t=1}^{M}w_{t}T_{t}$ is asymptotically
regular. Finally, by Theorem \ref{thm:5}, \textit{$\left\{ x^{k}\right\} _{k=0}^{\infty}$},
generated by the SA-DR Algorithm 1, converges weakly to a point

\begin{equation}
x^{*}\in\textup{Fix}\left(\sum_{t=1}^{M}w_{t}T_{t}\right)=\bigcap_{t=1}^{M}\textup{Fix}T_{t}=\bigcap_{t=1}^{M}\Gamma_{t}=\Gamma.
\end{equation}
Using the assumption of (\ref{eq:int}) that $\textup{int}\:\bigcap_{i\in I}C_{i}\neq\textrm{Ø}$
we apply \cite[Corollary 4.3.17(ii)]{Ceg-book} to all pairs of sets
in each string $\Gamma_{t}$ and obtain 
\begin{equation}
\Gamma=\bigcap_{i\in I}C_{i}.
\end{equation}
Therefore, $x^{*}\in\bigcap_{i\in I}C_{i}$. 

\noindent Since $\sum_{t=1}^{M}w_{t}T_{t}$ is a quasi-nonexpansive
operator, it is Fejér monotone. Therefore, \textit{$\left\{ x^{k}\right\} _{k=0}^{\infty}$},
generated by the SA-DR Algorithm 1 is Fejér monotone sequence. By
\cite[Proposition 5.10]{BC11}, the convergence of \textit{$\left\{ x^{k}\right\} _{k=0}^{\infty}$
}to $x^{*}$ is strong.
\end{proof}
Note that the operator $T_{t}$ is nonexpansive, but without the assumption
$\textup{int}\:\bigcap_{i\in I}C_{i}\neq\textrm{Ø}$ it need not be
$\Theta$-strictly quasi-nonexpansive which would have prevented us
from getting (\ref{eq:3.18}) from Theorem \ref{thm:2}(i), which
is a corner stone in the proof of Theorem 17.

For the Block-Iterative Douglas-Rachford algorithm we prove the following
convergence result.
\begin{thm}
\noindent \textbf{(BI-DR)}. Let $C_{1},C_{2},\ldots,C_{m}\subseteq\mathcal{H}$
be closed and convex sets with a nonempty intersection. For any $x^{0}\in\mathcal{H}$,
the sequence $\left\{ y^{k}\right\} _{k=0}^{\infty}$ of iterates
of the BI-DR Algorithm 2 with $I=I_{1}\cup I_{2}\cup\ldots\cup I_{M}$,
after full sweeps through all blocks, converges 

\noindent (i) weakly to a point $y^{*}$ such that $P_{C_{i_{\ell}^{t}}}(y^{*})\in\bigcap_{\ell=1}^{\gamma(t)}C_{i_{\ell}^{t}}$
for $\ell=1,2,\ldots,\gamma(t)$ and $t=1,2,\ldots,M,$ and 

\noindent (ii) strongly to a point $y^{*}$ such that $y^{*}\in\bigcap_{i=1}^{m}C_{i}$
if the additional assumption $\textup{int}\:\bigcap_{i\in I}C_{i}\neq\textrm{Ø}$
holds.\end{thm}
\begin{proof}
\noindent (i) Define \textbf{$Q_{t}:=\sum_{\ell=1}^{\gamma(t)-1}w_{\ell}^{t}T_{i_{\ell}^{t},i_{\ell+1}^{t}}+w_{\gamma(t)}^{t}T_{i_{\gamma(t)}^{t},i_{1}^{t}}$}
and let us look at the sequence
\begin{equation}
y^{k+1}=Q_{M}Q_{M-1}\cdots Q_{2}Q_{1}\left(y^{k}\right):=\left(\prod_{t=1}^{M}Q_{t}\right)\left(y^{k}\right),\label{eq:y-k}
\end{equation}

\noindent wherein the order of multiplication of operators is as indicated.
By \cite[Corollary 4.3.17(iv)]{Ceg-book}, the operators on the right-hand
sides of (\ref{eq:z-el-1}) and (\ref{eq:z-el-2}) are firmly nonexpansive,
for all $\ell=1,2,\ldots,\gamma(t)$, so, by Proposition \ref{prop:fne-is-ne},
they are nonexpansive. Each point in the nonempty intersection $\bigcap_{i\in I}C_{i}$,
is a fixed point for each 2-set-DR operator, thus, the operators on
the right-hand sides of (\ref{eq:z-el-1}) and (\ref{eq:z-el-2})
have a common fixed point. Then by Theorem \ref{thm:3}, for $w_{\ell}^{t}>0,$
for all $\ell=1,2,\ldots,\gamma(t)$, and $\sum_{\ell=1}^{\gamma(t)}w_{\ell}^{t}=1$,

\noindent 
\begin{equation}
\textup{Fix}Q_{t}=\left(\bigcap_{\ell=1}^{\gamma(t)-1}\textup{Fix}T_{i_{\ell}^{t},i_{\ell+1}^{t}}\right)\bigcap\textup{Fix}T_{i_{\gamma(t)}^{t},i_{1}^{t}}\supseteq\bigcap_{i=1}^{m}C_{i}\neq\textrm{Ø.}\label{eq:last-1}
\end{equation}

\noindent By Proposition \ref{prop:fne-is-ne} the operator $Q_{t}$
is firmly nonexpansive, thus, nonexpansive. Furthermore, by (\ref{eq:last-1})
it has a fixed point and, so, by Theorem \ref{thm:4} it is asymptotically
regular. By (\ref{eq:last-1}) 
\begin{equation}
\bigcap_{t=1}^{M}\left(\textup{Fix}Q_{t}\right)\supseteq\bigcap_{i=1}^{m}C_{i}\neq\textrm{Ø},
\end{equation}

\noindent so, by \cite[Lemma 2.3]{btam14}
\begin{equation}
\textup{Fix}\left(\prod_{t=1}^{M}Q_{t}\right)=\bigcap_{t=1}^{M}\left(\textup{Fix}Q_{t}\right).\label{eq:3.24}
\end{equation}
By \cite[Theorem 4.6]{bauschke2012compositions} the operator $\prod_{t=1}^{M}Q_{t}$
is asymptotically regular. Since the composition of firmly nonexpansive
operators is always nonexpansive, by Theorem \ref{thm:6-opial}, \textit{$\left\{ y^{k}\right\} _{k=0}^{\infty}$,}
generated by (\ref{eq:y-k}) and the BI-DR Algorithm 2, converges
weakly to a point $y^{*}$ such that: 
\begin{equation}
y^{*}\in\textup{Fix}\left(\prod_{t=1}^{M}Q_{t}\right).\label{eq:3.25}
\end{equation}
By (\ref{eq:3.25}), (\ref{eq:3.24}) and (\ref{eq:last-1}) we have:
\begin{equation}
y^{*}\in\bigcap_{t=1}^{M}\left(\left(\bigcap_{\ell=1}^{\gamma(t)-1}\textup{Fix}T_{i_{\ell}^{t},i_{\ell+1}^{t}}\right)\bigcap\textup{Fix}T_{i_{\gamma(t)}^{t},i_{1}^{t}}\right).\label{eq:3.26-1}
\end{equation}
By Lemma 15, $P_{C_{i_{\ell}^{t}}}\textup{Fix}T_{i_{\ell}^{t},i_{\ell+1}^{t}}=C_{i_{\ell}^{t}}\bigcap C_{i_{\ell+1}^{t}}$
for $\ell=1,2,\ldots,\gamma(t)-1,$ and $P_{C_{i_{\gamma(t)}^{t}}}\textup{Fix}T_{i_{\gamma(t)}^{t},i_{1}^{t}}=C_{i_{\gamma(t)}^{t}}\bigcap C_{i_{1}^{t}}.$
So, $P_{C_{i_{\ell}^{t}}}(y^{*})\in C_{i_{\ell+1}^{t}}$ for $\ell=1,2,\ldots,\gamma(t)-1,$
and $P_{C_{i_{\gamma(t)}^{t}}}(y^{*})\in C_{i_{1}^{t}}$. By the characterization
of projections, e.g., \cite[Theorem 1.2.4]{Ceg-book} we prove the
following: 
\begin{multline}
0\geq2\sum_{\ell=1}^{\gamma(t)-1}\left\langle y^{*}-P_{C_{i_{\ell+1}^{t}}}(y^{*}),P_{C_{i_{\ell}^{t}}}(y^{*})-P_{C_{i_{\ell+1}^{t}}}(y^{*})\right\rangle \\
+2\left\langle y^{*}-P_{C_{i_{1}^{t}}}(y^{*}),P_{C_{i_{\gamma(t)}^{t}}}(y^{*})-P_{C_{i_{1}^{t}}}(y^{*})\right\rangle \\
=\sum_{\ell=1}^{\gamma(t)-1}\left\Vert y^{*}-P_{C_{i_{\ell+1}^{t}}}(y^{*})\right\Vert ^{2}+\left\Vert y^{*}-P_{C_{i_{1}^{t}}}(y^{*})\right\Vert ^{2}-\sum_{\ell=1}^{\gamma(t)-1}\left\Vert y^{*}-P_{C_{i_{\ell}^{t}}}(y^{*})\right\Vert ^{2}\\
-\left\Vert y^{*}-P_{C_{i_{\gamma(t)}^{t}}}(y^{*})\right\Vert ^{2}+\sum_{\ell=1}^{\gamma(t)-1}\left\Vert P_{C_{i_{\ell+1}^{t}}}(y^{*})-P_{C_{i_{\ell}^{t}}}(y^{*})\right\Vert ^{2}+\left\Vert P_{C_{i_{1}^{t}}}(y^{*})-P_{C_{i_{\gamma(t)}^{t}}}(y^{*})\right\Vert ^{2}\\
=\sum_{\ell=1}^{\gamma(t)-1}\left\Vert P_{C_{i_{\ell+1}^{t}}}(y^{*})-P_{C_{i_{\ell}^{t}}}(y^{*})\right\Vert ^{2}+\left\Vert P_{C_{i_{1}^{t}}}(y^{*})-P_{C_{i_{\gamma(t)}^{t}}}(y^{*})\right\Vert ^{2}.\label{eq:3.26}
\end{multline}
Since the right-hand side of (\ref{eq:3.26}) is nonnegative it must
be equal to zero. So, we have $P_{C_{i_{\ell+1}^{t}}}(y^{*})=P_{C_{i_{\ell}^{t}}}(y^{*})$
for $\ell=1,2,\ldots,\gamma(t)-1,$ and $P_{C_{i_{1}^{t}}}(y^{*})=P_{C_{i_{\gamma(t)}^{t}}}(y^{*})$.
Therefore, $P_{C_{i_{\ell}^{t}}}(y^{*})\in\bigcap_{\ell=1}^{\gamma(t)}C_{i_{\ell}^{t}}$
for $t=1,2,\ldots,M.$ 

\noindent (ii) Continuing from (\ref{eq:3.26-1}), using the additional
assumption $\textup{int}\:\bigcap_{i\in I}C_{i}\neq\textrm{Ø}$ and
making repeated use of Lemma \ref{lem:16} yields, 
\begin{align}
y^{*} & \in\bigcap_{t=1}^{M}\left(\bigcap_{\ell=1}^{\gamma(t)}C_{i_{\ell}^{t}}\right)=\bigcap_{i=1}^{m}C_{i}.
\end{align}
 On the other hand, $\textup{\textit{\ensuremath{\bigcap_{i=1}^{m}C_{i}}}=Fix}\left(\prod_{t=1}^{M}Q_{t}\right)$
by applying the right-hand side expression of (\ref{eq:3.26-1}) and
Lemma \ref{lem:16}. Since $\prod_{t=1}^{M}Q_{t}$ is a nonexpansive
operator with a fixed point, it is quasi-nonexpansive \cite[Lemma 2.1.20]{Ceg-book},
and, therefore, it is Fejér monotone with respect to $\textup{Fix}\left(\prod_{t=1}^{M}Q_{t}\right),$i.e.,
with respect to $\bigcap_{i=1}^{m}C_{i}$ . Thus, $\left\{ y^{k}\right\} _{k=0}^{\infty}$,
generated by Algorithm 2 is a Fejér monotone sequence with respect
to $\bigcap_{i=1}^{m}C_{i}.$ By \cite[Proposition 5.10]{BC11}, the
convergence of $\left\{ y^{k}\right\} _{k=0}^{\infty}$ to $y^{*}$
is strong.
\end{proof}

\subsection{Special cases\label{sect:special-cases}}

For the String-Averaging Douglas-Rachford algorithm: (i) if all the
sets are included in one string then we obtain the Cyclic Douglas-Rachford
(CDR) algorithm of \cite[Section 3]{btam14}, (ii) if there are exactly
two sets in each string then we get an algorithm that can legitimately
be called the Simultaneous Douglas-Rachford (SDR) algorithm. This
SDR algorithm includes as a special case, when the weights are all
equal, the Averaged Douglas-Rachford algorithm of \cite[Theorem 3.3]{btam14}

For the Block-Iterative Douglas-Rachford algorithm: (i) if all the
sets are included in one block then we get the same Simultaneous Douglas-Rachford
(SDR) algorithm as above, including again the Averaged Douglas-Rachford
algorithm of \cite[Theorem 3.3]{btam14} as a special case. (ii) if
there are exactly two sets in each block, so that each two consecutive
blocks have a common set, then we get again the Cyclic Douglas-Rachford
(CDR) algorithm. 

Another case worth mentioning occurs if the initial point is included
in the first set of each string or each block. Then all the reflections
become orthogonal projections and our algorithms coincide with the
Strings-Averaging Projection (SAP) method \cite{ceh01} and with the
Block-Iterative Projection (BIP) method \cite{ac89}, respectively.
See also \cite[Corollary 3.1]{btam14}.

\section{A generalized $m$-set-Douglas-Rachford operator and algorithm\label{sect:third-generalization}}

In this section we propose a generalization of the 2-set-DR original
operator of Definition \ref{def:2-DR-1} that is applicable to $m$
sets and formulate an algorithmic structure to employ it. Instead
of reflecting consecutively into two sets and taking the midpoint
between the original point and the end-point of the two consecutive
reflections as the outcome of the 2-set-DR operator, we propose to
allow a finite number, say $r$ (greater or equal 2), of consecutive
reflections into $r$ sets and then taking the midpoint between the
original point and the end-point of the $r$ consecutive reflections
as the outcome of the newly defined operator. We name this as the
``generalized $r$-set-DR operator'', and show how it works algorithmically.
Before presenting those we need the following preliminary results.
\begin{prop}
\label{prop:reflection}\cite[Corollary 4.10]{BC11} Let $C$ be a
nonempty closed convex subset of $\mathcal{H}$. Then $\textup{Id}-P_{C}$
is firmly nonexpansive and $2P_{C}-\textup{Id}$ is nonexpansive.
\end{prop}
By \cite[Lemma 2.1.12]{Ceg-book} and \cite[Fig. 2.14]{Ceg-book}
any composition of nonexpansive (NE) operators is NE and any convex
combination of NE operators is NE.
\begin{thm}
\cite[Theorem 2.2.10(i)-(iii)]{Ceg-book}\label{thm:dr=00003Dfne}
Let $T:X\rightarrow\mathcal{H}$. Then the following conditions are
equivalent:
\begin{enumerate}
\item $T$ is firmly nonexpansive.
\item $T_{\lambda}$ is nonexpansive for any $\lambda\in[0,2]$.
\item $T$ has the form $T=\frac{1}{2}(S+\textup{Id})$, where $S:X\rightarrow\mathcal{H}$
is a nonexpansive operator.
\end{enumerate}
\end{thm}
Our generalized Douglas-Rachford operator is presented in the following
definition.
\begin{defn}
\label{def:m-dr}Let $C_{1},C_{2},\ldots,C_{m}\subseteq\mathcal{H}$
be nonempty closed convex sets. For $r=2,3,\ldots,m$ $(m\geq2)$
define the composite reflection operator $V_{C_{1},C_{2},\ldots,C_{r}}:\mathcal{H\rightarrow\mathcal{H}}$
by

\begin{equation}
V_{C_{1},C_{2},\ldots,C_{r}}:=R_{C_{r}}R_{C_{r-1}}\cdots R_{C_{1}}.\label{eq:composite ref.}
\end{equation}
The generalized $r$-set-DR operator $T_{C_{1},C_{2},\ldots,C_{r}}:\mathcal{H\rightarrow\mathcal{H}}$
is defined by 
\begin{equation}
T_{C_{1},C_{2},\ldots,C_{r}}:=\frac{1}{2}\left(\textup{Id}+V_{C_{1},C_{2},\ldots,C_{r}}\right).\label{eq:r-sets-DR op.}
\end{equation}

\end{defn}
For $r=2$ the generalized $r$-set-DR operator coincides with the
$2$-set-DR operator. For $r=3$ the generalized $r$-set-DR operator
coincides with the $3$-set-DR iteration defined in \cite[Eq. (2)]{artacho2013recent}.

We will make use of the following corollary which extends \cite[Corollary 4.3.17(ii)]{Ceg-book}.
\begin{cor}
\label{corollary 21-1}Let $C_{1},C_{2},\ldots,C_{m}\subseteq\mathcal{H}$
be nonempty closed convex sets with a nonempty intersection, and let
$V_{C_{1},C_{2},\ldots,C_{m}}:\mathcal{H\rightarrow\mathcal{H}}$
and $T_{C_{1},C_{2},\ldots,C_{m}}$ be as in Definition \ref{def:m-dr}
If $\textup{int}\:\bigcap_{i=1}^{m}C_{i}\neq\textrm{Ø}$ then 
\begin{equation}
\bigcap_{i=1}^{m}C_{i}=\textup{Fix}T_{C_{1},C_{2},\ldots,C_{m}}.
\end{equation}
\end{cor}
\begin{proof}
It is clear that $\bigcap_{i=1}^{m}\textup{int}C_{i}=\bigcap_{i=1}^{m}\textup{int}\,\textup{Fix}R_{C_{i}}\subseteq\bigcap_{i=1}^{m}\textup{Fix}R_{C_{i}}.$
By \cite[Proposition 2.1.41]{Ceg-book} $R_{C_{i}},$ for $i=1,2,\ldots,m,$
are $C$-strictly quasi-nonexpansive. Theorem \ref{thm:2}(ii) and
the fact that $\textup{Fix}R_{C_{i}}=C_{i},$ for $i=1,2,\ldots,m$,
yield 
\begin{equation}
\textup{Fix}T_{C_{1},C_{2},\ldots,C_{m}}=\textup{Fix}V_{C_{1},C_{2},\ldots,C_{m}}=\bigcap_{i=1}^{m}\textup{Fix}R_{C_{i}}=\bigcap_{i=1}^{m}C_{i}.
\end{equation}

\end{proof}
Next we present the algorithm that uses generalized $r$-set-DR operators
and prove its convergence.\medskip{}

\noindent \textbf{Algorithm 3.}

\noindent \textbf{Initialization}: $x^{0}\in\mathcal{H}$ is arbitrary.

\noindent \textbf{Iterative Step:} Given the current iterate $x^{k}$,
calculate, for all $r=2,3,\ldots,m,$

\begin{equation}
x^{k+1}=\sum_{r=2}^{m}w_{r}T_{C_{1},C_{2},\ldots,C_{r}}(x^{k})
\end{equation}

\noindent with $w_{r}>0,$ for all $r=2,3,\ldots,m$, and $\sum_{r=2}^{m}w_{r}=1$
.
\begin{thm}
\label{thm:23}Let $C_{1},C_{2},\ldots,C_{m}\subseteq\mathcal{H}$
be nonempty closed convex sets with a nonempty intersection. if 
\begin{equation}
\textup{int}\:\bigcap_{i\in I}C_{i}\neq\textrm{Ø}
\end{equation}
 then for any $x^{0}\in\mathcal{H}$, any sequence $\left\{ x^{k}\right\} _{k=0}^{\infty}$,
generated by Algorithm 3, converges strongly to a point $x^{*}\in\bigcap_{i\in I}C_{i}$. \end{thm}
\begin{proof}
By Proposition \ref{prop:reflection}, the reflection $R_{C_{i}}$
for all $i=1,2,\ldots,m$, is nonexpansive operator. By the facts
noted after Proposition \ref{prop:reflection}, the operator $V_{C_{1},C_{2},\ldots,C_{r}}$
defined in (\ref{eq:composite ref.}) is nonexpansive. By Theorem
\ref{thm:dr=00003Dfne}(i) and (iii), the operator $T_{C_{1},C_{2},\ldots,C_{r}}$
defined in (\ref{eq:r-sets-DR op.}) is firmly nonexpansive, and by
Proposition \ref{prop:fne-is-ne} it is nonexpansive. Since $\bigcap_{i=1}^{m}C_{i}\neq\textrm{Ø}$,
any point in the intersection is a fixed point of the reflection $R_{C_{i}}$
for all $i=1,2,\ldots,m$, and such point is also a fixed point of
any operator $T_{C_{1},C_{2},\ldots,C_{r}}$ for all $r=2,3,\ldots,m$.
By Theorem \ref{thm:3} for $w_{r}>0,$ for all $r=2,3,\ldots,m$,
and $\sum_{r=2}^{m}w_{r}=1$,
\begin{equation}
\textup{Fix}\left(\sum_{r=2}^{m}w_{r}T_{C_{1},C_{2},\ldots,C_{r}}\right)=\bigcap_{r=2}^{m}\textup{Fix}T_{C_{1},C_{2},\ldots,C_{r}}\supseteq\bigcap_{i=1}^{m}C_{i}\neq\textrm{Ø}.\label{eq:Fix of r-sets-DR}
\end{equation}

Using again Proposition \ref{prop:fne-is-ne}, the operator $\sum_{r=2}^{m}w_{r}T_{C_{1},C_{2},\ldots,C_{r}}$
is firmly nonexpansive, thus, nonexpansive, and has a fixed point
according to (\ref{eq:Fix of r-sets-DR}). So, by Theorem \ref{thm:4},
it is asymptotically regular, therefore, by Theorem \ref{thm:6-opial},
\textit{$\left\{ x^{k}\right\} _{k=0}^{\infty}$,} generated by Algorithm
3, converges weakly to a point $x^{*}\in\mathcal{H}$ for which \textit{
\begin{equation}
x^{*}\in\textup{Fix}\left(\sum_{r=2}^{m}w_{r}T_{C_{1},C_{2},\ldots,C_{r}}\right).
\end{equation}
}By (\ref{eq:Fix of r-sets-DR}) and Corollary \ref{corollary 21-1}
we have
\begin{equation}
\textup{Fix}\left(\sum_{r=2}^{m}w_{r}T_{C_{1},C_{2},\ldots,C_{r}}\right)=\bigcap_{r=2}^{m}\textup{Fix}T_{C_{1},C_{2},\ldots,C_{r}}=\bigcap_{r=2}^{m}\left(\bigcap_{i=1}^{r}C_{i}\right)=\bigcap_{d=1}^{m}C_{d},
\end{equation}
therefore, $x^{*}\in\bigcap_{i\in I}C_{i}$. To prove that the convergence
is strong recall that $\sum_{r=2}^{m}w_{r}T_{C_{1},C_{2},\ldots,C_{r}}$
is firmly nonexpansive, thus nonexpansive, according to \cite[Theorem 2.2.4]{Ceg-book},
with a fixed point. Therefore, it is quasi-nonexpansive \cite[Lemma 2.1.20]{Ceg-book},
and thus is Fejér monotone . Hence, \textit{$\left\{ x^{k}\right\} _{k=0}^{\infty}$},
generated by the Algorithm 3, is a Fejér monotone sequence . By \cite[Proposition 5.10]{BC11},
the convergence of \textit{$\left\{ x^{k}\right\} _{k=0}^{\infty}$
}to $x^{*}$ is strong.
\end{proof}
If we replace the reflections by projections, we get the convergence
to a point in the intersection of the sets, because the algorithm
then becomes a special case of string-averaging projections (SAP),
see, \cite{censor2013convergence}.

\noindent \medskip{}

\textbf{Acknowledgments}. We thank Andrzej Cegielski for having read
carefully our work and for his insightful comments. We thank Rafa\l{}
Zalas for a discussion on an earlier draft of this paper. We greatly
appreciate the insightful and constructive comments of two anonymous
referees and the Associate Editor which helped us improve the paper.
This work was supported by Research Grant No. 2013003 of the United
States-Israel Binational Science Foundation (BSF).

\bibliographystyle{amsplain}
\addcontentsline{toc}{section}{\refname}\bibliography{Bib-rafiq-paper1-24-11-15}

\providecommand{\bysame}{\leavevmode\hbox to3em{\hrulefill}\thinspace}
\providecommand{\MR}{\relax\ifhmode\unskip\space\fi MR }
% \MRhref is called by the amsart/book/proc definition of \MR.
\providecommand{\MRhref}[2]{%
  \href{http://www.ams.org/mathscinet-getitem?mr=#1}{#2}
}
\providecommand{\href}[2]{#2}
\begin{thebibliography}{10}

\bibitem{ac89}
R.~Aharoni and Y.~Censor, \emph{{Block-iterative projection methods for
  parallel computation of solutions to convex feasibility problems}}, Linear
  Algebra and its Applications \textbf{120} (1989), 165--175.

\bibitem{artacho2013douglas}
F.~J.~A. Artacho, J.~M. Borwein, and M.~K. Tam, \emph{Douglas-{R}achford
  feasibility methods for matrix completion problems}, The {ANZIAM} Journal
  \textbf{55} (2014), 299--326.

\bibitem{artacho2013recent}
\bysame, \emph{Recent results on {D}ouglas-{R}achford methods for combinatorial
  optimization problems}, Journal of Optimization Theory and Applications
  \textbf{163} (2014), 1--30.

\bibitem{artacho2013global}
F.~J.~Arag{\'o}n Artacho and J.~M. Borwein, \emph{Global convergence of a
  non-convex {D}ouglas-{R}achford iteration}, Journal of Global Optimization
  \textbf{57} (2013), 753--769.

\bibitem{bauschke2003composition}
H.~H. Bauschke, \emph{The composition of projections onto closed convex sets in
  {H}ilbert space is asymptotically regular}, Proceedings of the American
  Mathematical Society \textbf{131} (2003), 141--146.

\bibitem{bb96}
H.~H. Bauschke and J.~M. Borwein, \emph{{On projection algorithms for solving
  convex feasibility problems}}, SIAM Review \textbf{38} (1996), 367--426.

\bibitem{BC11}
H.~H. Bauschke and P.~L. Combettes, \emph{{Convex Analysis and Monotone
  Operator Theory in Hilbert Spaces}}, Springer, New York, NY, USA, 2011.

\bibitem{bauschke2004finding}
H.~H. Bauschke, P.~L. Combettes, and D.~R. Luke, \emph{Finding best
  approximation pairs relative to two closed convex sets in {H}ilbert spaces},
  Journal of Approximation Theory \textbf{127} (2004), 178--192.

\bibitem{bauschke2006strongly}
\bysame, \emph{A strongly convergent reflection method for finding the
  projection onto the intersection of two closed convex sets in a {H}ilbert
  space}, Journal of Approximation Theory \textbf{141} (2006), 63--69.

\bibitem{bauschke2012compositions}
H.~H. Bauschke, V.~Mart{\'\i}n-M{\'a}rquez, S.~M. Moffat, and X.~Wang,
  \emph{Compositions and convex combinations of asymptotically regular firmly
  nonexpansive mappings are also asymptotically regular}, Fixed Point Theory
  and Applications \textbf{2012:53} (2012), 1--11.

\bibitem{bauschke2014local}
H.~H. Bauschke and D.~Noll, \emph{On the local convergence of the
  {D}ouglas-{R}achford algorithm}, Archiv der Mathematik \textbf{102} (2014),
  589--600.

\bibitem{bauschke2015linear}
H.~H. Bauschke, D.~Noll, and H.~M. Phan, \emph{Linear and strong convergence of
  algorithms involving averaged nonexpansive operators}, Journal of
  Mathematical Analysis and Applications \textbf{421} (2015), 1--20.

\bibitem{borwein2011douglas}
J.~M. Borwein and B.~Sims, \emph{The {D}ouglas-{R}achford algorithm in the
  absence of convexity}, {Fixed-Point Algorithms for Inverse Problems in
  Science and Engineering} (H.~H. Bauschke, R.~S. Burachik, P.~L. Combettes,
  V.~Elser, D.~R. Luke, and H.~Wolkowicz, eds.), Springer, 2011, pp.~93--109.

\bibitem{btam14}
J.~M. Borwein and M.~K. Tam, \emph{{A cyclic {D}ouglas-{R}achford iteration
  scheme}}, Journal of Optimization Theory and Applications \textbf{160}
  (2014), 1--29.

\bibitem{borwein2014reflection}
\bysame, \emph{Reflection methods for inverse problems with application to
  protein conformation determination}, arXiv preprint arXiv:1408.4213 (2014).

\bibitem{borwein2013cyclic}
\bysame, \emph{The cyclic {D}ouglas-{R}achford method for inconsistent
  feasibility problems}, Nonlinear Convex Analysis \textbf{16} (2015),
  537--584.

\bibitem{butnariu1990behavior}
D.~Butnariu and Y.~Censor, \emph{On the behavior of a block-iterative
  projection method for solving convex feasibility problems}, International
  Journal of Computer Mathematics \textbf{34} (1990), 79--94.

\bibitem{byrne1995block}
C.~L. Byrne, \emph{Block-iterative methods for image reconstruction from
  projections}, IEEE Transactions on Image Processing \textbf{5} (1995),
  792--794.

\bibitem{Ceg-book}
A.~Cegielski, \emph{{Iterative Methods for Fixed Point Problems in Hilbert
  Spaces}}, Lecture Notes in Mathematics 2057, Springer-Verlag, Berlin,
  Heidelberg, Germany, 2012.

\bibitem{censor2002block}
Y.~Censor and T.~Elfving, \emph{Block-iterative algorithms with diagonally
  scaled oblique projections for the linear feasibility problem}, SIAM Journal
  on Matrix Analysis and Applications \textbf{24} (2002), 40--58.

\bibitem{ceh01}
Y.~Censor, T.~Elfving, and G.~T. Herman, \emph{{Averaging strings of sequential
  iterations for convex feasibility problems}}, {Inherently Parallel Algorithms
  in Feasibility and Optimization and Their Applications} (D.~Butnariu,
  Y.~Censor, and S.~Reich, eds.), Elsevier Science Publishers, Amsterdam, The
  Netherlands, 2001, pp.~101--114.

\bibitem{censor2009string}
Y.~Censor and A.~Segal, \emph{On the string averaging method for sparse common
  fixed-point problems}, International Transactions in Operational Research
  \textbf{16} (2009), 481--494.

\bibitem{censor2003convergence}
Y.~Censor and E.~Tom, \emph{Convergence of string-averaging projection schemes
  for inconsistent convex feasibility problems}, Optimization Methods and
  Software \textbf{18} (2003), 543--554.

\bibitem{censor2013convergence}
Y.~Censor and A.~J. Zaslavski, \emph{Convergence and perturbation resilience of
  dynamic string-averaging projection methods}, Computational Optimization and
  Applications \textbf{54} (2013), 65--76.

\bibitem{censor2008diagonally}
Y.~Censorr, T.~Elfving, G.~T. Herman, and T.~Nikazad, \emph{On diagonally
  relaxed orthogonal projection methods}, SIAM Journal on Scientific Computing
  \textbf{30} (2008), 473--504.

\bibitem{combettes2004solving}
P.~L. Combettes, \emph{Solving monotone inclusions via compositions of
  nonexpansive averaged operators}, Optimization \textbf{53} (2004), 475--504.

\bibitem{combettes2011proximal}
P.~L. Combettes and J.~C. Pesquet, \emph{Proximal splitting methods in signal
  processing}, {Fixed-Point Algorithms for Inverse Problems in Science and
  Engineering} (H.~H. Bauschke, R.~S. Burachik, P.~L. Combettes, V.~Elser,
  D.~R. Luke, and H.~Wolkowicz, eds.), Springer, 2011, pp.~185--212.

\bibitem{gordon2005component}
D.~Gordon and R.~Gordon, \emph{Component-averaged row projections: A robust,
  block-parallel scheme for sparse linear systems}, SIAM Journal on Scientific
  Computing \textbf{27} (2005), 1092--1117.

\bibitem{karonis2013distributed}
N.~T. Karonis, K.~L. Duffin, C.~E. Ordo{\~n}ez, B.~Erdely, T.~D. Uram, E.~C.
  Olson, G.~Coutrakon, and M.~E. Papka, \emph{Distributed and hardware
  accelerated computing for clinical medical imaging using proton computed
  tomography (p{CT})}, Journal of Parallel and Distributed Computing
  \textbf{73} (2013), 1605--1612.

\bibitem{svaiter2011weak}
B.~F. Svaiter, \emph{On weak convergence of the {D}ouglas-{R}achford method},
  SIAM Journal on Control and Optimization \textbf{49} (2011), 280--287.

\end{thebibliography}

\end{document}